\documentclass[journal]{IEEEtran}

\ifCLASSINFOpdf
\else
\fi
\usepackage{graphics} 
\usepackage{epsfig} 
\usepackage{amsmath} 
\usepackage{amssymb}  
\usepackage{amsthm}

\usepackage{bm}
\usepackage{multirow}
\usepackage{tabularx}
\usepackage{algpseudocode}
\usepackage{algorithm}

\usepackage{pgfplots}
\usepackage{epstopdf}

\usepackage{color}
\usepackage[caption=false, font=footnotesize]{subfig}

\allowdisplaybreaks

\DeclareMathOperator{\atan2}{atan2}
\DeclareMathOperator{\diag}{diag}
\DeclareMathOperator{\re}{Re}
\DeclareMathOperator{\im}{Im}

\newtheorem{theorem}{Theorem}
\newtheorem{proposition}{Proposition}
\newtheorem{assumption}{Assumption}

\theoremstyle{remark} \newtheorem{rmk}{Remark}

\newcommand{\ud}{\,\mathrm{d}}
\begin{document}
%
\title{A New Voltage Stability-Constrained Optimal Power Flow Model: Sufficient Condition, SOCP Representation, and Relaxation}
%
%
%

\author{Bai~Cui,~\IEEEmembership{Student Member,~IEEE}
        and~Xu~Andy~Sun,~\IEEEmembership{Senior Member,~IEEE}
\thanks{B. Cui is with the School of Electrical and Computer Engineering, Georgia Institute of Technology, Atlanta, GA 30332 USA {\tt\small bcui7@gatech.edu}.}
\thanks{X. A. Sun is with the Department of Industrial and Systems Engineering, Georgia Institute of Technology, Atlanta, GA 30332 USA {\tt\small andy.sun@isye.gatech.edu}.}}

\maketitle

\begin{abstract}
This paper proposes a novel voltage stability-constrained optimal power flow (VSC-OPF) model utilizing a recently proposed sufficient condition on power flow Jacobian nonsingularity.
We show that this condition is second-order conic representable when load powers are fixed. Through the incorporation of the convex sufficient condition {and thanks to the recent development of convex relaxation of OPF models}, we {cast} a VSC-OPF formulation as a second-order cone program (SOCP). An approximate model is introduced to improve the scalability of the formulation to larger systems. Extensive computation results on \textsc{Matpower} and NESTA instances confirm the effectiveness and efficiency of the formulation.
\end{abstract}


%
\IEEEpeerreviewmaketitle

\section{Introduction} \label{sect1}


\IEEEPARstart{T}{he} {need to ensure steady-state voltage stability and maintain sufficient loading margin in optimal power flow (OPF) models has led to the development of voltage stability-constrained OPF (VSC-OPF) models, which solves OPF problems while accounting for voltage stability limits at the same time.} Traditionally, to avoid system instability, security constraints such as voltage magnitude limits and line flow limits are enforced in normal {OPF models}. However, the effectiveness of these security constraints {alone} in safeguarding system stability {may be insufficient} in modern power systems {with adequate reactive power support, which is demonstrated by a two-bus example in \cite{Todescato16}}. Another motivation for the inclusion of steady-state stability limit in an OPF formulation is the increasing trend to operate power systems ever closer to their operational limits due to increased demand and competitive electricity market. Without stability constraints, the robustness of the OPF solution against voltage instability is not ensured. 

{In \cite{Canizares01}, two sets of power flow equations representing system at base case and critical condition are present in an OPF model. The power injections of the two sets of equations are related by a loading factor, which is used to represent the loading margin to voltage instability to be optimized. The model is extended to a multi-objective one in \cite{Milano03} in which voltage stability and social welfare are simultaneously taken care of. An extension to incorporate $N-1$ contingencies in this VSC-OPF model has been reported in \cite{Milano05} where a heuristic contingency ranking technique is applied for the sake of computation tractability. An alternative method to account for contingencies in a VSC-OPF model based on iterative CPF-OPF computation is presented in \cite{Milano05b}. However, the loading margin is only quantified along one direction of power variation in these models. Dynamic voltage stability has been considered in security-constrained OPF such that systems subject to contingencies will settle down to stable operating points. Dynamic simulation with scenario filtering techniques have been employed to this end in \cite{Capitanescu09, Capitanescu11}. These methods are highly dependent on the selection of contingencies and suffer from scalability issue. A different strategy to represent proximity to voltage instability is through the use of minimum singular value (MSV) of the power flow Jacobian, which can be used as a stability constraint in an VSC-OPF model. The main drawbacks of the method are that 1) the value of MSV can hardly be interpreted in terms of loading margin; 2) MSV is not an explicit function of the optimization variables. Linearization and iterative algorithms have been proposed trying to address the second issue \cite{Kodsi07, Avalos08}. However, the computational cost is prohibitively high for large-scale systems.} 

{To circumvent the weaknesses of the aforementioned VSC-OPF models and achieve a better trade-off between robustness and computational tractability, several heuristic voltage stability indicators have been proposed to be embedded in VSC-OPF formulations. For instance, the $L$-index originally proposed in \cite{PK86} has been used as an indicator for voltage stability improvement in \cite{Kumar14}. Leveraging semidefinite programming relaxation of OPF,  this problem can be formulated as a semidefinite program with quasi-convex objective \cite{Pederson15}. Polyhedron approximation of security boundaries has been applied in a DC-OPF model in \cite{Canizares16} for proper accounting of system loading margins. However, the characterization of security boundary gets complicated as the dimension of feasible region goes up. 
In \cite{Bolognani16}, a sufficient condition for existence and uniqueness of high-voltage solution in distribution system is obtained using fixed-point argument, which has been extended in \cite{Yu15}. Similar techniques have subsequently been applied to yield stronger results in \cite{Simpson16} and \cite{WangC16}. The sufficient voltage stability condition in \cite{Simpson16} has been used for voltage stability improvement and `voltage stress minimization' for reactive power flow equations in \cite{Todescato16}. A voltage stability index based on branch flow is integrated in VSC-OPF formulation in \cite{Kamwa14}. Major concerns of these indices are their conservativeness and computational properties. Hence, the main motivation of this paper is to apply a novel and tight voltage stability index in the VSC-OPF model which enjoys nice computational properties under very mild approximation.}

{We first introduce a sufficient condition for power flow Jacobian nonsingularity that we proposed recently in \cite{Wang17}}. We then formulate a VSC-OPF problem in which the voltage stability margin is quantified by the condition. We show that when load powers are fixed, this voltage stability condition describes a second-order conic representable  set in a transformed voltage space. Thus second-order cone program (SOCP) reformulation can naturally incorporate the condition. Notice that the formulation does not require the DC or decoupled power flow assumptions. To improve computation time, we sparsify the dense stability constraints while preserving very high accuracy.

The rest of the paper is organized as follows. Section \ref{sect2} provides background on power system modeling. The sufficient condition for power flow Jacobian nonsingularity is introduced in Section \ref{sect:condition}. We discuss the VSC-OPF formulation, its convex reformulation, and sparse approximation in Section \ref{sect3}. Section \ref{sect4} presents results of extensive computational experiments {and comparative studies. Section \ref{sect6} concludes.}

\section{Background} \label{sect2}

\subsection{Notations}
The cardinality of a set or the absolute value of a (possibly) complex number is denoted by $|\cdot|$. $\mathrm{i} = \sqrt{-1}$ is the imaginary unit. $\mathbb{R}$ and $\mathbb{C}$ are the set of real and complex numbers, respectively. For vector $x \in \mathbb{C}^n$, $\| x\|_p$ denotes the $p$-norm of $x$ where $p \geq 1$ and $\diag(x) \in \mathbb{C}^{n \times n}$ is the associated diagonal matrix. The $n$-dimensional identity matrix is denoted by $\mathbf{I}_n$. $\mathbf{0}_{n \times m}$ denotes an $n\times m$ matrix of all 0's.  For $A \in \mathbb{C}^{n\times n}$, $A^{-1}$ is the inverse of $A$. For $B \in \mathbb{C}^{m\times n}$, $B^T$, $B^H$ are respectively the transpose and conjugate transpose of $B$, and $B^*$ is the matrix with complex conjugate entries. The real and imaginary parts of $B$ are denoted as $\re B$ and $\im B$. {$b_i$ denotes the vector formed by the $i$th row of $B$.}

\subsection{Power system modeling} \label{pfmodel}
We consider a connected single-phase power system with $n+m$ buses operating in steady-state. The underlying topology of the system can be described by an undirected connected graph $G = (\mathcal{N}, \mathcal{E})$, where $\mathcal{N} = \mathcal{N}_G \cup \mathcal{N}_L$ is the set of buses equipped with ($\mathcal{N}_G$) and without ($\mathcal{N}_L$) generators (or generator buses and load buses), and that $|\mathcal{N}_G| = m$ and $|\mathcal{N}_L| = n$. We number the buses such that the set of load buses are $\mathcal{N}_L = \{ 1, \ldots, n \}$ and the set of generator buses are $\mathcal{N}_G = \{ n+1, \ldots, n+m \}$. Generally, for a complex matrix $A \in \mathbb{C}^{(n+m)\times k}$, define $A_L = (A_{ij})_{i \in \mathcal{N}_L}$. That is, $A_L$ is the first $n$ rows of the matrix $A$. Similarly, define $A_G = (A_{ij})_{i \in \mathcal{N}_G}$. Every bus $i$ in the system is associated with a voltage phasor $V_i = |V_i| e^{\mathrm{i}\theta_i}$ where $|V_i|$ and $\theta_i$ are the magnitude and phase angle of the voltage. We will find it convenient to adopt rectangular coordinates for voltages sometimes, so we also define $V_i = e_i + \mathrm{i}f_i$. The generator buses are modeled as PV buses, while load buses are modeled as PQ buses. For bus $i$, the injected power is given as $S_i = P_i + \mathrm{i}Q_i$. 

The line section between buses $i$ and $j$ in the system is weighted by its complex admittance $y_{ij} = 1/z_{ij} = g_{ij} + \mathrm{i}b_{ij}$. The nodal admittance matrix $Y = G + \mathrm{i}B \in \mathbb{C}^{(n+m)\times(n+m)}$ has components $Y_{ij} = -y_{ij}$ and $Y_{ii} = y_{ii} + \sum_{j=1}^{n+m} y_{ij}$ where $y_{ii}$ is the shunt admittance at bus $i$.

The nodal admittance matrix relates system voltages and currents as
\begin{equation}
	\begin{bmatrix}
		I_L \\
		I_G
	\end{bmatrix} = 
	\begin{bmatrix}
		Y_{LL} & Y_{LG} \\
		Y_{GL} & Y_{GG}
	\end{bmatrix}
	\begin{bmatrix}
		V_L \\
		V_G
	\end{bmatrix}.
	\label{admit}
\end{equation}
We obtain from (\ref{admit}) that
\begin{equation}
	V_L = -Y_{LL}^{-1}Y_{LG}V_G + Y_{LL}^{-1}I_L.
	\label{equi1}
\end{equation}
Define the vector of equivalent voltage to be $E = -Y_{LL}^{-1} Y_{LG} V_G$ and the impedance matrix to be $Z = Y_{LL}^{-1}$ (we assume the invertibility of $Y_{LL}$ and note that this is almost always the case for practical systems). With the definitions, ({\ref{equi1}}) can be rewritten as 
\begin{equation} \label{eq:nodekvl}
{V_L = E + ZI_L.}
\end{equation}

For practical power systems, the generator buses have regulated voltage magnitudes and small phase angles. It is common in voltage stability analysis to assume that the generator buses have constant voltage phasor $V_G$ \cite{Wang17, PK86}. The assumption can be partially justified by the fact that voltage instability are mostly caused by system overloading due to excess demand at load side, irrelevant of generator voltage variations.
\begin{assumption} \label{as1}
	The vector of generator bus voltages $V_G$ is constant.
\end{assumption}

Note that Assumption \ref{as1} is always satisfied for uni-directional distribution systems where the only source is modeled as a slack bus with fixed voltage phasor. {The voltage stability constraint in the paper is based on our recent result on the nonsingularity of power flow Jacobian \cite{Wang17}. The derivation of the result takes advantage of the special characteristics systems with constant generator voltage vector $E$. With Assumption 1, $E$ is fixed and the result in \cite{Wang17} can be applied.}

The power flow equations in the rectangular form relate voltages and power injections at each bus $i \in \mathcal{N}$ via
\begin{subequations} \label{eq:recpf}
	\begin{align}
		P_i &= \sum_{j=1}^{n+m} [G_{ij}(e_ie_j + f_if_j) + B_{ij}(e_jf_i - e_if_j)], \label{eq:recpf:real}\\
		Q_i &= \sum_{j=1}^{n+m} [G_{ij}(e_jf_i - e_if_j) - B_{ij}(e_ie_j + f_if_j)].\label{eq:recpf:react}
	\end{align}
\end{subequations}

\begin{rmk}
{The power flow Jacobian with Assumption 1 is given by
\begin{equation} \label{eq:pfjac:red}
J_{LL} := \begin{bmatrix}
\frac{\partial P_L}{\partial e_L} & \frac{\partial P_L}{\partial f_L} \\
\frac{\partial Q_L}{\partial e_L} & \frac{\partial Q_L}{\partial f_L}
\end{bmatrix}.
\end{equation}
Note that $J_{LL}$ is in fact a submatrix of the full Jacobian considering generator real power equations:
\begin{equation}
J = \begin{bmatrix}
J_{GG} & J_{GL} \\
J_{LG} & J_{LL}
\end{bmatrix}.
\end{equation}}

{As we know, voltage stability studies are primarily concerned with the singularity of power flow Jacobian $J$. Of course, for generic matrix $J$, the singularity of its principal submatrices are not necessarily related to that of the full matrix. Then the assumption of constant generator voltage phasors seems to be questionable since the stability analysis based on a submatrix may not be relevant. However, we note this is not the case in voltage stability analysis. First of all, the validity of using power flow Jacobian as a voltage stability indicator is based on the assumption that $\det J_{LL} \neq 0$. In this case the system stability is determined by the reduced Jacobian $J_{\mathrm{red}} = J_{GG} - J_{GL}J_{LL}^{-1}J_{LG}$, whose determinant is singular if and only if the determinant of the power flow Jacobian
\begin{equation}
\det J = \det J_{LL} \det (J_{GG} - J_{GL}J_{LL}^{-1}J_{LG})
\end{equation} 
is singular \cite[Chap. 5]{Cutsem08}. However, the singularity of $J_{LL}$ is itself one of the  mechanisms of voltage collapse, which is called singularity-induced bifurcation and has been demonstrated through a rudimentary dynamic power system model in \cite{Venkat92}. Second, the singularity of $J$ is often associated with the ill-conditioning of the matrix $J_{LL}$; and the MSV of $J_{LL}$ tends to decrease monotonically with increased loading levels, which are demonstrated by IEEE 9-bus system in Fig. \ref{fig:sigma9bus}. Therefore we believe the study of $J_{LL}$ for voltage stability purposes can be justified from both physical and numerical perspective.}

\begin{figure}[!t]
	\centering
	\includegraphics[width=2.8in]{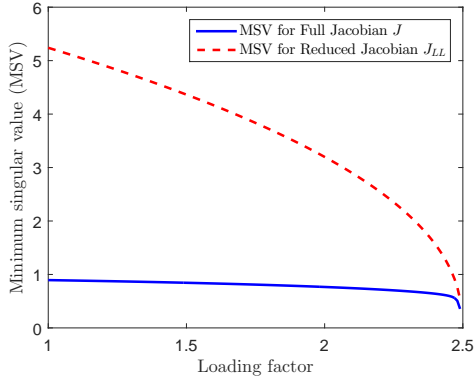}
	\caption{MSVs of full and reduced power flow Jacobian with respect to system loading for 9-bus system.}
	\label{fig:sigma9bus}
\end{figure}

\end{rmk}

\begin{rmk}
{After the overexcitation limiter of a generator takes effect, the terminal voltage of the generator can no longer be regulated, and a common modeling practice is to switch the bus type from PV to PQ. We note that the generator can also be modeled as a constant excitation emf behind synchronous impedance based on \cite[Sect. 3.4.2]{Cutsem08}, the validity of which has been justified in \cite{Wang13}. The synchronous impedance can be absorbed by the system admittance matrix and the model reduces to the one with constant voltage sources and constant power load buses. This can be done iteratively every time generators reaches their reactive power limits after OPF computation.}
\end{rmk}

\subsection{AC-OPF formulation}

Using the power flow equations \eqref{eq:recpf:real}-\eqref{eq:recpf:react}, a standard AC-OPF model can be written as 
\begin{subequations}\label{eq:ACOPF}
	\begin{align}
		\mathrm{min} \quad
		& \sum_{i \in \mathcal{N}_G} f_i(P_{G_i}) \label{eq:ACOPF:obj}\\
		\mathrm{s.t.} \quad
		& P_i(e, f) = P_{G_i} - P_{D_i}, &&  i \in \mathcal{N} \label{eq:ACOPF:realpf} \\
		& Q_i(e, f) = Q_{G_i} - Q_{D_i}, &&  i \in \mathcal{N}  \label{eq:ACOPF:reactpf}\\
		& \underline{P}_{G_i} \leq P_{G_i} \leq \overline{P}_{G_i}, &&  i \in \mathcal{N}_G \label{eq:ACOPF:realg} \\
		& \underline{Q}_{G_i} \leq Q_{G_i} \leq \overline{Q}_{G_i}, &&  i \in \mathcal{N}_G \label{eq:ACOPF:reactg} \\
		& \underline{V}_i^2 \leq e_i^2 + f_i^2 \leq \overline{V}_i^2, &&  i \in \mathcal{N} \label{eq:ACOPF:vol} \\
		& |P_{ij}(e, f)| \leq \overline{P}_{ij}, &&  (i,j) \in \mathcal{E} \label{eq:ACOPF:pbr} \\
		& |I_{ij}(e, f)| \leq \overline{I}_{ij}, &&  (i,j) \in \mathcal{E}, \label{eq:ACOPF:curr}
	\end{align}
\end{subequations}
where $f_i(P_{G_i})$ in \eqref{eq:ACOPF:obj} is the variable production cost of generator $i$, assuming to be a convex quadratic function; $P_{G_i}$ and $P_{D_i}$ in \eqref{eq:ACOPF:realpf}-\eqref{eq:ACOPF:reactpf} are the real power generation and load at bus $i$, respectively; $Q_{G_i}$ and $Q_{D_i}$ are the reactive power generation and load at bus $i$; $P_i(e,f)$ and $Q_i(e,f)$ are given by the power flow equations \eqref{eq:recpf}; 
constraints \eqref{eq:ACOPF:realg}-\eqref{eq:ACOPF:reactg} represent the real and reactive power generation capability of generator $i$.
$P_{ij}$ and $I_{ij}$ in \eqref{eq:ACOPF:pbr}-\eqref{eq:ACOPF:curr} are the real power and current magnitude flowing from bus $i$ to $j$ for line $(i,j) \in \mathcal{E}$, respectively.

\section{A Sufficient Condition for Nonsingularity of Power flow Jacobian} \label{sect:condition}
A sufficient condition for the nonsingularity of power flow Jacobian is recently proposed in \cite{Wang17} as stated in the following theorem. We will use this result {to derive a voltage stability index which is to be embedded in an OPF model to form a VSC-OPF formulation.}
\begin{theorem}\label{thm:condition}
	The power flow Jacobian of (\ref{eq:recpf}) is nonsingular if
	\begin{equation} \label{condition}
		|V_i| - \|z_i^T \diag(I_L) \|_1 > 0 , \quad i \in \mathcal{N}_L.
	\end{equation}
\end{theorem}
{The proof is based on similarity transformation of the power flow Jacobian. We have shown that the transformed matrix is strictly diagonally dominant as long as \eqref{condition} holds. Since strictly diagonally dominant matrices are nonsingular and similarity transformation preserves eigenvalues, the power flow Jacobian is nonsingular when \eqref{condition} holds. The proof takes advantage of the special structure of the matrix $J_{LL}$. Under Assumption 1, the power flow Jacobian and $J_{LL}$ coincide. For proof of the theorem, see \cite{Wang17}.}

{The term $\|z_i^T \diag(I_L) \|_1$ in Theorem 1 can be thought of as the generalized voltage drop between the equivalent source with voltage $E_i$ to the load. Then the theorem states that the system is voltage stable if the generalized voltage drop is less than the corresponding load voltage magnitude for all load buses. It has been shown in \cite{Wang17} that the result is strong, meaning that the violation of the condition is often immediately followed by the loss of voltage stability.}

{It is suggested in \cite{PK86} that the following condition is satisfied at the point of voltage instability under certain simplifying assumptions (proportional load current variations, etc.)
\begin{equation}
	\left| \sum_{i=1}^n Z_{ji}I_i \right| = |V_j|. \label{eq:condPK}
\end{equation}
It is seen from \eqref{eq:nodekvl} that the left hand side is the voltage drop between the equivalent source with voltage $E_j$ and the load. The result implies that under certain assumptions, the voltage stability of a multi-bus system resembles that of a two-bus system where the voltage stability boundary is achieved when the magnitude of voltage drop and load voltage are identical. Due to various assumptions, the condition works relatively well under proportional load variations, but becomes less effective as load variation deviates from the assumed proportional pattern.}

{We note the similarity between the condition \eqref{eq:condPK} and \eqref{condition} used in the paper. The condition \eqref{condition} is weaker in the sense that the generalized voltage drop $\|z_i^T \diag(I_L) \|_1$ is larger than the actual voltage drop in \eqref{eq:condPK}, but it nevertheless generalizes the latter condition and does not require the proportional current injection assumption. For a more thorough comparison of the two conditions, see \cite{Wang17}.}

\section{A New Model for VSC-OPF} \label{sect3}

The standard AC-OPF formulation embeds system security constraints as line real power and current limits in (\ref{eq:ACOPF:pbr}) and (\ref{eq:ACOPF:curr}). However, the parameters in these security-related constraints, such as $\overline{P}_{ij}$ and $\overline{I}_{ij}$, are calculated off-line using possible dispatch scenarios that do not necessarily represent the actual system conditions \cite{Milano03}. This motivates the formulation of VSC-OPF models. In this section, we propose a new model for VSC-OPF using the voltage condition derived in \eqref{condition} and show that it has nice convex properties amenable for efficient computation.

\subsection{New formulation}
We propose the following new VSC-OPF model,
\begin{subequations}
	\label{eq:newVSCOPF}
	\begin{align}
		\mathrm{min} \quad
		& \sum_{i \in \mathcal{N}_G} f_i(P_{G_i}) \\
		\mathrm{s.t.} \quad
		& \text{\eqref{eq:ACOPF:realpf} -- \eqref{eq:ACOPF:vol}} \nonumber\\
		& |V_i| - \sum_{j = 1}^n \frac{A_{ij}}{|V_j|} \geq \underline{t}_{i}, \quad i \in \mathcal{N}_L. \label{eq:newVSCOPF:cond}
	\end{align} 
\end{subequations}
where $A_{ij} := |Z_{ij}S_j|$.
The key constraint is \eqref{eq:newVSCOPF:cond}, which reformulates the left-hand side of \eqref{condition} by writing line currents as the ratio of apparent powers that satisfy the power flow equations \eqref{eq:recpf} and voltages, and $\underline{t}_i$ is a preset positive parameter to control the level of voltage stability. We note that line flow constraints are not included in the VSC-OPF formulation \eqref{eq:newVSCOPF}. We have deliberately chosen not to include them since 1) we would like to demonstrate the capability of the proposed voltage stability constraint in restraining system margins to voltage instability, and 2) we believe the proposed constraint is better suited for stability constraining purposes. To guarantee the same level of voltage stability, line flow constraints come at a price of higher level of conservativeness compared to the proposed stability constraint since the line flow constraints are not intrinsic voltage stability measures. It then follows that to ensure similar level of voltage stability, the inclusion of line flow constraints shrinks the feasibility region of the problem. Of course, there are no technical difficulties in the inclusion of line flow constraints in our formulation and we agree that for lines with low thermal ratings or low line flow margins, the inclusion of corresponding constraints are necessary and beneficial. To ensure that \eqref{eq:newVSCOPF} is a proper formulation with good computational property, we first show that the set of voltages satisfying condition \eqref{eq:newVSCOPF:cond} is voltage stable, and then we show that \eqref{eq:newVSCOPF:cond} is second-order cone (SOC) representable, thus convex, when $S_L$ is constant. The condition of constant $S_L$ is always met in OPF problems.

\subsubsection{Connectedness}
A necessary condition for voltage instability is the singularity of power flow Jacobian {\cite[Sect. 7.1.2]{Cutsem08}}. Assume that the zero injection solution of power flow equations \eqref{eq:recpf} is voltage stable with a nonsingular Jacobian (which always holds for any physically meaningful system). We know from \eqref{eq:recpf} that every entry of $J$ is a continuous function of voltages, so the eigenvalues of $J$ are also continuous in voltages. Since a continuous function maps a connected set to another connected set, if a given connected set of power flow solutions contains the zero injection solution (which is voltage stable) and the corresponding power flow Jacobian of every point in the set is nonsingular, then the set characterizes a subset of voltage stable solutions. Define the set $\mathcal{S}_0 := \{V_L | \text{ $\eqref{condition}$ holds}\}$ and $\mathcal{S}_0 \supseteq \mathcal{S}_{\underline{t}} := \{V_L | \text{ $\eqref{eq:newVSCOPF:cond}$ holds}\}$. We know from Theorem \ref{thm:condition} that the power flow Jacobian is nonsingular for $V_L \in \mathcal{S}_0$, we also know the zero injection solution is in $\mathcal{S}_0$. Therefore, in order to show the set $\mathcal{S}_{\underline{t}}$ is voltage stable, we show the more general case that $\mathcal{S}_0$ is voltage stable, which amounts to showing the connectedness of $\mathcal{S}_0$. We give the proof of this property below.

\begin{theorem}
	The set $\mathcal{S}_0$ is connected.
\end{theorem}
\begin{proof}
	To show the set is connected, we fix a point in the set and show that for any other point in the set, the line segment between the two points lies in the set. 
		
	When load currents are all zero, it follows from \eqref{eq:nodekvl} that the nodal load voltages are simply $E$. We denote the zero injection voltage solution by $v_0$, that is, $v_0 := E$. Then $v_0 \in \mathcal{S}_0$ follows immediately since $\Vert z_i^T \diag(I_L) \Vert_1 = 0$ for all $i\in \mathcal{N}_L$. Take $v_1 \in \mathcal{S}_0$ and define $V_L$ parametrized by $t \in \left[0,1\right]$ as $V_L(t) = v_0 + \left(v_1 - v_0\right)t$. We will show $V_L(t)$ is in $\mathcal{S}_0$. It is clear that current injections are linear functions of $t$, since we know from \eqref{eq:nodekvl} that
	\begin{subequations}
		\begin{align}
			I_L(t) &= Y_{LL} \left(V_L(t) - E\right) \label{eq:ilparametrized:a} \\
			&= Y_{LL} \left(v_0 + \left(v_1-v_0\right)t - v_0 \right) \label{eq:ilparametrized:b} \\
			&= Y_{LL} \left(v_1-v_0\right)t. \label{eq:ilparametrized:c}
		\end{align}
	\end{subequations}
	We claim that for any $t \in [0,1]$ the derivative of $\sum_{j=1}^n \left|Z_{ij}I_j\right|$ is larger than or equal to the magnitude of that of $|V_i|$ for all $i \in \mathcal{N}_L$. Since current injections are linear in $t$, let $Z_{ij}I_j$ be denoted by $a_{ij}t + \mathrm{i}b_{ij}t$ for real numbers $a_{ij}$ and $b_{ij}$ for all $(i,j) \in \mathcal{N}_L \times \mathcal{N}_L$, and denote $a := \sum_{j=1}^n a_{ij}$ and $b := \sum_{j=1}^n b_{ij}$ for brevity, then for each $i \in \mathcal{N}_L$ we have
	\begin{subequations} \label{eq:derivative:zi}
		\begin{align}
			\frac{\mathrm{d}}{\mathrm{d}t} \left(\sum_{j=1}^n \left|Z_{ij}I_j\right| \right)
			&= \frac{\mathrm{d}}{\mathrm{d}t} \left(\sum_{j=1}^n \sqrt{a_{ij}^2 + b_{ij}^2} \right)t \\
			&= \sum_{j=1}^n \sqrt{a_{ij}^2 + b_{ij}^2} \\
			&\geq \sqrt{a^2 + b^2},
		\end{align} 
	\end{subequations} 
	where the inequality is due to successive application of trigonometric inequality. On the other hand, the voltage magnitude $|V_i|$ is
	\begin{multline}
	|V_i| = \left|v_{0,i} + z_i^TI_L\right| \\
	= \textstyle \sqrt{\left(\re (v_{0,i}) + \sum_{j=1}^n a_{ij}t\right)^2 + \left(\im (v_{0,i}) + \sum_{j=1}^n b_{ij}t\right)^2},
	\end{multline}
	and the derivative of $|V_i|$ with respect to $t$ is
	\begin{align}
		\frac{\mathrm{d}|V_i|}{\mathrm{d}t} = \frac{a(\re (v_{0,i}) + at) + b(\im (v_{0,i}) + bt)}{\sqrt{(\re (v_{0,i}) + at)^2 + (\im (v_{0,i}) + bt)^2}}.
	\end{align}
	Then, by Cauchy-Schwarz inequality we have $| \mathrm{d}|V_i| / \mathrm{d}t | \leq \sqrt{a^2 + b^2}$. Comparing with \eqref{eq:derivative:zi}, we see the claim holds.
	
	Suppose $\sum_{j=1}^n|Z_{ij}I_j(t_1)| \ge |V_i(t_1)|$ for some $t_1 \in (0,1)$ and $i \in \mathcal{N}_L$, then based on the fundamental theorem of calculus we have
	\begin{subequations}
		\begin{align}
			\sum_{j=1}^n|Z_{ij}I_j(1)| &= \sum_{j=1}^n|Z_{ij}I_j(t_1)| + \int_{t_1}^1 \left(\sum_{j=1}^n \left|Z_{ij}I_j\right| \right)' \ud t \\
			&\ge \sum_{j=1}^n|Z_{ij}I_j(t_1)| + \sqrt{a^2 + b^2}(1-t_1),
		\end{align} 
	\end{subequations}
	and
	\begin{subequations}
		\begin{align}
			|V_i(1)| &= |V_i(t_1)| + \int_{t_1}^1 |V_i(t)|' \ud t \\
			&\le |V_i(t_1)| + \sqrt{a^2 + b^2}(1-t_1).
		\end{align} 
	\end{subequations}
	The two inequalities imply that $\sum_{j=1}^n|Z_{ij}I_j(1)| \ge |V_i(1)|$, which is a contradiction since $v_1 \in \mathcal{S}_0$. We conclude the line segment between $v_0$ and $v_1$ lies in $\mathcal{S}_0$.
\end{proof}

\subsubsection{SOC representation of voltage stability constraint}
The voltage stability constraint \eqref{eq:newVSCOPF:cond} is not directly a convex constraint in the voltage variable $V_i$, however, we show that it can be reformulated as a convex constraint, more specifically, an SOC constraint in squared voltage magnitude $|V_i|^2$ providing $S_L$ is fixed. This SOC reformulation will be utilized in the following section for SOCP relaxation of VSC-OPF. 
\begin{proposition}\label{prop:VSCSOCP}
	Constraint \eqref{eq:newVSCOPF:cond} is SOC representable in the squared voltage magnitude $|V_i|^2$'s, i.e. \eqref{eq:newVSCOPF:cond} can be reformulated using SOC constraints in $|V_i|^2$'s.
\end{proposition}
\begin{proof}
	First of all, introduce variable $c_{ii} := |V_i|^2$, and $x_i, z_i$ for each bus $i \in \mathcal{N}_L$ such that
	\begin{align}
		x_i & \leq \sqrt{c_{ii}}, \label{8}\\
		x_iz_i & \geq 1, \label{9} \\
		x_i & \geq 0. \notag
	\end{align}
	Note that $x_iz_i = (\frac{x_i+z_i}{2})^2 - (\frac{x_i-z_i}{2})^2$ and $c_{ii} = (\frac{c_{ii}+1}{2})^2 - (\frac{c_{ii}-1}{2})^2$, then we see both (\ref{8}) and (\ref{9}) can be rewritten as the following SOC constraints
	\begin{align}
		\sqrt{x_i^2 + \frac{(c_{ii}-1)^2}{4}} & \leq \frac{c_{ii}+1}{2}, \\
		\sqrt{1 + \frac{(x_i - z_i)^2}{4}} & \leq \frac{x_i + z_i}{2}.
	\end{align}
	Therefore, by defining $A_{ij} = |Z_{ij}S_j|$, \eqref{eq:newVSCOPF:cond} can be equivalently represented as
	\begin{subequations}\label{eq:VSconstrReform}
		\begin{align}
			& x_i - \sum_{j=1}^n A_{ij}z_j \geq \underline{t}_i, \label{eq:condLinear}\\
			&\left\|[x_i, (c_{ii}-1)/2]^T\right\|_2 \le (c_{ii}+1)/2 , \label{eq:VScr2}\\
			& \left\|[1, (x_i - z_i)/2]^T\right\|_2 \leq (x_i + z_i)/{2}, \label{eq:VScr3}\\
			& x_i \geq 0, \label{eq:VScr4}
		\end{align}
	\end{subequations}
	for every bus $i \in \mathcal{N}_L$, which are SOCP constraints. 
\end{proof}

\subsection{SOCP relaxation of VSC-OPF}
By Proposition \ref{prop:VSCSOCP}, the voltage stability condition \eqref{condition} is reformulated as SOCP constraints \eqref{eq:VSconstrReform}. However, the power flow equations \eqref{eq:ACOPF:realpf}-\eqref{eq:ACOPF:reactpf} are still nonconvex. In the following, we propose an SOCP relaxation of the proposed VSC-OPF model \eqref{eq:newVSCOPF} by combining the SOC reformulation of the voltage stability constraint \eqref{eq:VSconstrReform} with the recent development of SOCP relaxation of standard AC-OPF \cite{Kocuk2016SOCP}. In particular, for each line $(i,j) \in \mathcal{E}$, define 
\begin{subequations}\label{eq:cs-def}
	\begin{align}
		c_{ij} &= e_ie_j + f_if_j \label{eq:c-def}\\
		s_{ij} &= e_if_j - e_jf_i. \label{eq:s-def}
	\end{align}
\end{subequations}
An implied constraint of \eqref{eq:c-def}-\eqref{eq:s-def} is the following:
\begin{align} \label{eq:cs-eq}
	c_{ij}^2 + s_{ij}^2 = c_{ii}c_{jj}.
\end{align}
Now we can introduce the following SOCP relaxation of the VSC-OPF model \eqref{eq:newVSCOPF} in the new variables $c_{ii}$, $c_{ij}$, and $s_{ij}$ as follows
\begin{subequations} \label{eq:VSC-OPF-SOCP}
	\begin{align}
		\mathrm{min} \quad
		& \sum_{i \in \mathcal{N}_G} f_i(P_{G_i}) \notag \\
		\mathrm{s.t.} \quad
		& P_{G_i} - P_{D_i} = G_{ii}c_{ii} + \sum_{j \in N(i)} P_{ij}, \ i \in \mathcal{N} \label{eq:P-cs}\\
		& Q_{G_i} - Q_{D_i} = -B_{ii}c_{ii} + \sum_{j \in N(i)} Q_{ij}, i \in \mathcal{N} \label{eq:Q-cs} \\
		& \underline{V}_i^2 \leq c_{ii} \leq \overline{V}_i^2, \quad i \in \mathcal{N} \label{eq:V-cs}\\
		& c_{ij} = c_{ji}, \ s_{ij} = -s_{ji}, \quad (i,j) \in \mathcal{E} \label{eq:cs}\\
		& c_{ij}^2 + s_{ij}^2 \le c_{ii}c_{jj} \quad (i,j) \in \mathcal{E}  \label{eq:cs-square}\\
		& \eqref{eq:ACOPF:realg},\eqref{eq:ACOPF:reactg}, \eqref{eq:VSconstrReform} \notag
	\end{align}
\end{subequations}
where the power flow equations \eqref{eq:ACOPF:realpf}-\eqref{eq:ACOPF:reactpf} are rewritten in the $c,s$ variables as \eqref{eq:P-cs} and \eqref{eq:Q-cs}. $N(i)$ denotes the set of buses adjacent to bus $i$. The line real and reactive powers are $P_{ij} = G_{ij}c_{ij} - B_{ij}s_{ij}$ and $Q_{ij} = -G_{ij}s_{ij} - B_{ij}c_{ij}$. The nonconvex constraint \eqref{eq:cs-eq} is relaxed as \eqref{eq:cs-square}, which can be easily written as an SOCP constraint as $\|[c_{ij}, s_{ij}, (c_{ii} - c_{jj})/2]^T\|_2 \leq {(c_{ii}+c_{jj})}/{2}$. \eqref{eq:V-cs} is a linear constraint in the square voltage magnitude $c_{ii}$. Notice that the SOCP formulation of the voltage stability constraint \eqref{eq:VSconstrReform} is not a relaxation, but an exact formulation of the original voltage {stability} condition \eqref{eq:newVSCOPF:cond}, and it fits nicely into the overall SOCP relaxation of the VSC-OPF model \eqref{eq:VSC-OPF-SOCP}. We have employed the basic SOCP relaxation of the AC-OPF in \eqref{eq:VSC-OPF-SOCP}. There are many ways to strengthen the relaxation, see \cite{Kocuk2016SOCP} for a few formulations. The main advantage of the adopted formulation lies in its speed, which may proven crucial for certain online applications. On the other hand, the main point we try to convey in the paper regarding the convex formulation is that the proposed voltage stability constraint is in fact second-order cone representable. This simple fact means that the constraint can be integrated in any other SOCP relaxation as well.

\subsection{Sparse approximation of SOCP relaxation}

Due to the density of stability condition \eqref{eq:condLinear}, the computation times of the VSC-OPF formulation \eqref{eq:VSC-OPF-SOCP} are significantly longer than normal OPF especially for larger power systems. The differences in computation time can be observed from Table I, where it is seen that for large IEEE instances, VSC-OPF is much slower than AC-OPF. The term `density' refers to the fact that each voltage stability constraint in \eqref{eq:condLinear} is coupled with almost all load buses since the matrix $A$ in \eqref{eq:condLinear} is dense. This is to be contrasted with the power flow equations or line flow constraints where the admittance matrices are sparse and power injection of a bus is only a function of its voltage phasor as well as those of its neighboring buses. Fig. \ref{fig:sparse300bus} shows the sparsity pattern of the matrix $A$ for IEEE 300-bus system, it can be seen from Fig. \ref{fig:sparse300bus}a that almost all entries are nonzero even though most of them are very small. To better discern the relative magnitude of the entries, we set entries less than $5\times 10^{-4}$ to zero in Fig. \ref{fig:sparse300bus}b, now the heat map becomes much sparser which suggests that a majority of entries are indeed small ($<5\times 10^{-4}$). Therefore, in order to speed up computation, we can approximate most of the entries by constants without sacrificing too much accuracy. 


\begin{figure} 
	\centering
	\subfloat[Orignal sparsity pattern.]{%
		\includegraphics[width=0.5\linewidth]{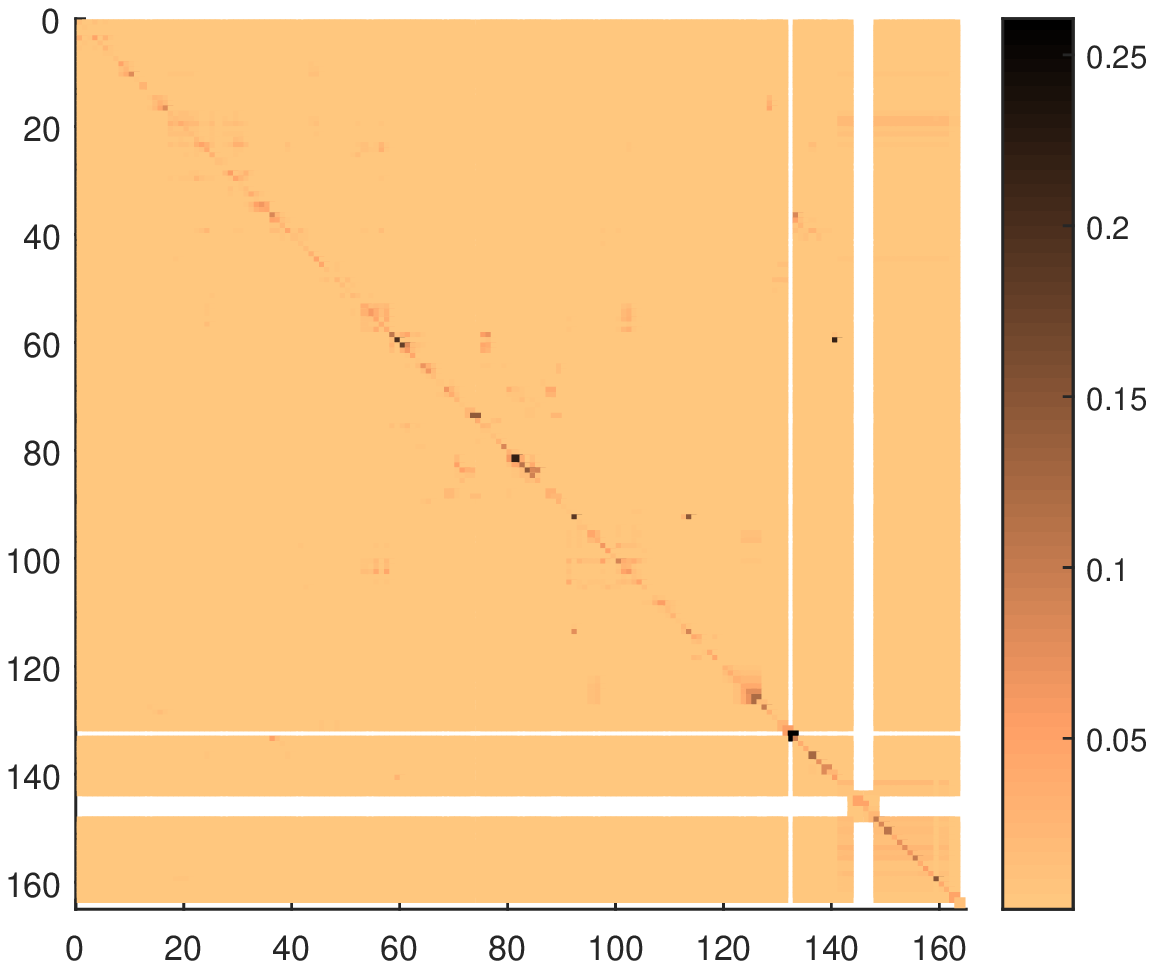}}
	\label{fig:sparse300a}\hfill
	\subfloat[Sparsity pattern with entries less than $5\times 10^{-4}$ set to zero.]{%
		\includegraphics[width=0.5\linewidth]{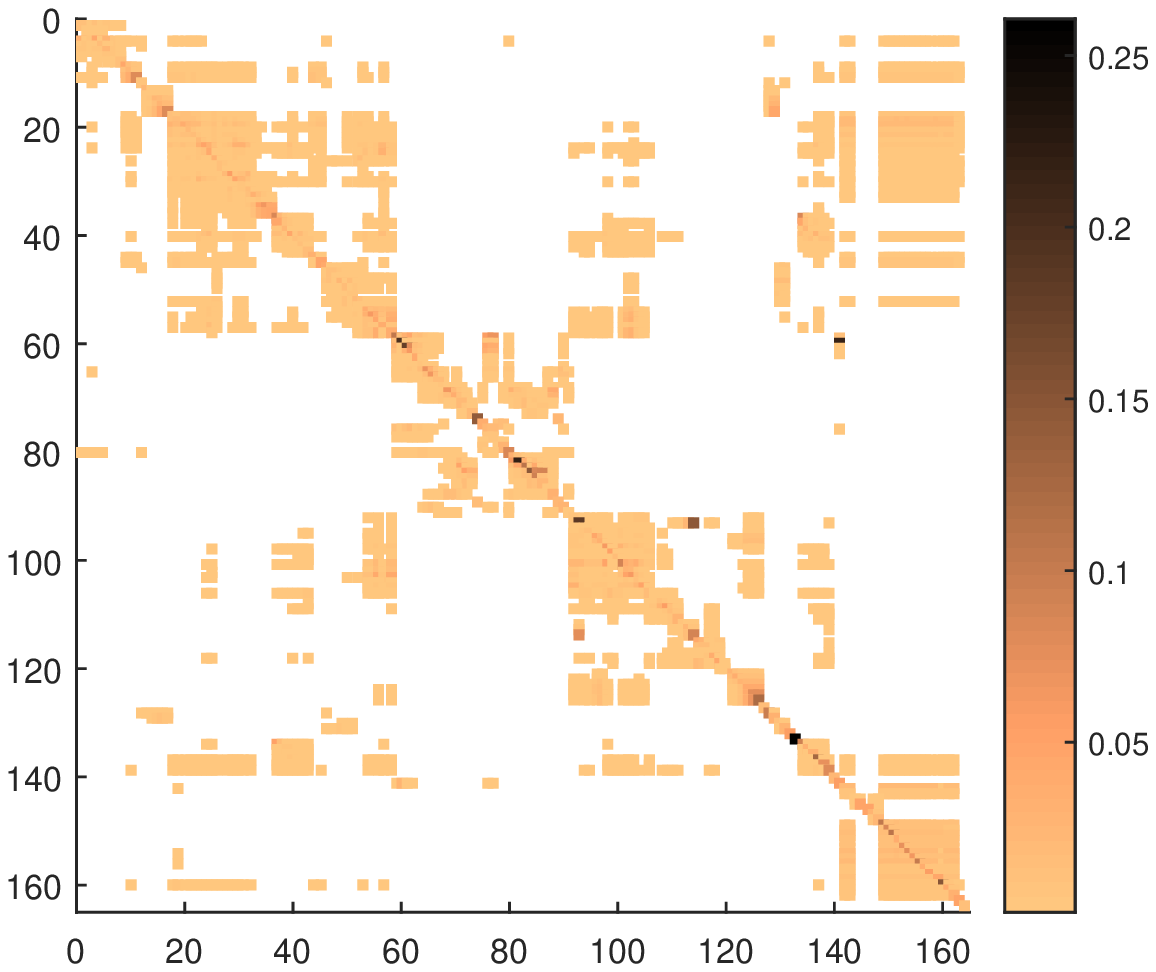}}
	\label{fig:sparse300b}
	\caption{Sparsity pattern of matrix $A$ for IEEE 300-bus system.}
	\label{fig:sparse300bus} 
\end{figure}

The first step of the approximation is to approximate the coeffcient matrix $A$ of the stability constraints by a sparse matrix $\tilde A$. To illustrate our approach of sparse approximation, we  rewrite the linear constraint \eqref{eq:condLinear} in matrix-vector form as
\begin{equation} \label{eq:condMatrix}
x - Ay \geq \underline{t}.
\end{equation} 
Then the approach to construct the sparse approximate matrix $\tilde A$ can be summarized as in Algorithm \ref{alg:sparsify}.

\begin{algorithm}
\caption{Sparse approximation of $A$}
\begin{algorithmic}
\State {$\gamma \gets \gamma_0 \quad $} 
\Comment {initialize tunable sparsity parameter}
\State {$\tilde A \gets \mathbf{0}_{n\times n}$} 
\Comment {initialize $\tilde A$}
\For {$1 \leq i \leq n$}
	\State {$RS \gets \sum_j A_{ij} $}
	\Comment {compute $i$th row sum of matrix $A$}
	\While {$\sum_j \tilde A_{ij} < \gamma RS$}
		\State {$j_{\max} \gets \arg \max a_i$}
		\State {$\tilde A_{i,j_{\max}} \gets A_{i,j_{\max}}$}
		\State {$A_{i,j_{\max}} \gets 0$}
	\EndWhile
\EndFor
\end{algorithmic}
\label{alg:sparsify}
\end{algorithm}
Simply put, for each row of matrix $A$, Algorithm \ref{alg:sparsify} constructs the corresponding row of the approximate matrix $\tilde A$ by ignoring all elements except the largest ones whose sum amounts to more than $\gamma$ of the total row sum. We notice that the element $Z_{ij}$ of the impedance matrix can be understood as the coupling intensity measure between buses $i$ and $j$. Thanks to the sparsity of practical power systems, each bus is only strongly coupled with its neighboring buses and weakly coupled with most other buses. Therefore, the matrix $\tilde A$ is generally sparse. We notice a similar approximation has been applied to the $L$-index in the context of PMU allocation \cite{Pordanjani13}. The connection between $L$-index and the proposed stability condition has been discussed in Section \ref{sect:condition} and more extensively in \cite{Wang17}.

Then \eqref{eq:condMatrix} can be approximated by
\begin{equation} \label{eq:sparseCond}
x - \tilde Ay \geq \underline{t} + \Delta a / \overline{V},
\end{equation}
where $\Delta a \in \mathbb{R}^n$ is the row sum difference between $A$ and $\tilde A$ that is defined as $\Delta a_i = \sum(a_i - \tilde a_i)$ and $\overline{V} = \max\{\overline{V}_i \mid i \in \mathcal{N}_L \}$. We have thus obtained the sparse VSC-OPF formulation which is identical to \eqref{eq:VSC-OPF-SOCP} except the stability constraint \eqref{eq:condLinear} is replaced by \eqref{eq:sparseCond}. The new formulation is presented as
\begin{subequations} \label{eq:sparse_VSC-OPF-SOCP}
	\begin{align}
	\mathrm{min} \quad
	& \sum_{i \in \mathcal{N}_G} f_i(P_{G_i}) \notag \\
	\mathrm{s.t.} \quad
	& P_{G_i} - P_{D_i} = G_{ii}c_{ii} + \sum_{j \in N(i)} P_{ij}, \ i \in \mathcal{N} \label{eq:P-scs}\\
	& Q_{G_i} - Q_{D_i} = -B_{ii}c_{ii} + \sum_{j \in N(i)} Q_{ij}, i \in \mathcal{N} \label{eq:Q-scs} \\
	& \underline{V}_i^2 \leq c_{ii} \leq \overline{V}_i^2, \quad i \in \mathcal{N} \label{eq:V-scs}\\
	& c_{ij} = c_{ji}, \ s_{ij} = -s_{ji}, \quad (i,j) \in \mathcal{E} \label{eq:scs}\\
	& c_{ij}^2 + s_{ij}^2 \le c_{ii}c_{jj} \quad (i,j) \in \mathcal{E}  \label{eq:scs-square}\\
	& \eqref{eq:ACOPF:realg},\eqref{eq:ACOPF:reactg}, \eqref{eq:VScr2}\text{--}\eqref{eq:VScr4}, \eqref{eq:sparseCond} \notag
	\end{align}
\end{subequations}

We notice that feasibility of problem \eqref{eq:sparse_VSC-OPF-SOCP} is implied by the feasibility of the original problem \eqref{eq:VSC-OPF-SOCP}. To see this we only need to focus on \eqref{eq:sparseCond} and \eqref{eq:condLinear}, from which we have
\begin{align*}
x - Ay & \leq x - \tilde A y - (\Delta a)y_{\min} \\
& \leq x - \tilde A y - \Delta a / \overline{V},
\end{align*}
where the last inequality comes from \eqref{8}, \eqref{9}, \eqref{eq:V-cs}.

\section{Computational Experiments} \label{sect4}

\begin{table*}[!t]
	\centering
	\caption{Results Summary for Standard IEEE Instances.}
	\begin{tabular}{ |c||c|c||c|c|c|c|c|c|}
		\hline
		\multirow{2}{*}{Test Case} & \multicolumn{2}{c||}{Cost (\$/h)} & \multirow{2}{*}{OG (\%)} &  \multirow{2}{*}{$\underline{t}$} & \multirow{2}{*}{$t_a^{AC}$} & \multirow{2}{*}{$\Delta \lambda_{\max}^{AC}$  (\%)} & \multirow{2}{*}{$\Delta \sigma_{\min}^{AC} (\%)$} & \multirow{2}{*}{DS (\%)}\\
		\cline{2-3}
		& AC & SOCP & & & & & & \\
		\hline
		\hline
		case24\_ieee\_rts & $64059.32$ & $63344.99$ &  $1.12$ & $0.86$ & $0.86$ & $0.12$ & $0.16$ & $0.08$	\\
		\hline
		case30 & $577.16$ & $574.90$ & $0.39$ & $0.97$ & $0.97$ & $5.02$ & $0.00$ & $0.07$ \\
		\hline
		case\_ieee30 & $9985.41$ & $9220.51$ & $7.66$  & $0.88$ & $0.88$ & $7.92$ & $3.75$ & $0.60$	\\
		\hline
		case39 & $43667.91$ & $42552.76$ &  $2.55$  & $0.83$ & $0.83$ & $6.49$ & $0.32$ & $0.48$ \\
		\hline
		case57 & $41737.79$ & $41710.91$ &  $0.06$ & $0.66$ & $0.66$ & $0.02$ & $0.02$ & $0.31$	\\
		\hline
		case89pegase & $5849.28$ & $5810.12$ & $0.67$ & $0.72$ & $0.72$ & $2.22$ & $0.21$ & $2.61$ \\
		\hline
		case118 & $130009.61$ & $129385.66$ & $0.48$ & $0.98$ & $0.98$ & $-0.21$ & $0.33$ & $0.44$	\\
		\hline
		case300 & $724935.75$ & $718655.31$ &   $0.87$  & $0.29$ & $0.29$	& $-0.30$ & $1.13$ & $1.03$ \\
		\hline
		case1354pegase & $74062.27$ & $74000.28$ & $0.08$ & $0.64$ & $0.64$ & $0.87$ & $0.00$ & $0.93$ \\
		\hline
		case2383wp & $1857927.67$ & $1846897.40$ & $0.59$ & $0.77$ & $0.77$ & $0.00$ & $0.00$ & $1.64$ \\
		\hline
		\textbf{average} & ------ & ------ & \bm{$1.45$} & \bm{$0.76$} & \bm{$0.76$} & \bm{$1.99$} & \bm{$0.59$} & \bm{$0.82$} \\
		\hline
	\end{tabular}
	\label{tb:IEEE}
\end{table*}

\begin{figure*}[!t]
	\centering
	\begin{tikzpicture}
	\begin{axis}[
	ybar,
	bar width=.3cm,
	width=\textwidth,
	height=.4\textwidth,
	tick label style={font=\small},
	tickpos=left,
	xticklabels={24,30fr, 30ieee, 73, 89, 189, 1354, 1394, 1397, 1460, 2224, 2383, 2736, 2737}, 
	xtick={1,2,3,4,5,6,7,8,9,10,11,12,13,14},
	ymin=-0.5,
	legend entries={$\Delta \sigma_{\min}$ (\%), $\Delta \lambda_{\max}$ (\%)},
	y tick label style={/pgf/number format/.cd,%
		scaled y ticks = false,
		set thousands separator={},
		fixed
	},
	]
	\addplot +[bar shift=-.2cm, area legend] coordinates {
		(1,1.614138294
		) (2,1.219284547
		) (3,2.817172315
		) (4,2.375154227
		) (5,1.221122272
		) (6,6.927146766
		) (7,1.29628E-06
		) (8,1.10514337
		) (9,-7.02824E-07
		) (10,1.734565891
		) (11,1.486477192
		) (12,4.966346089
		) (13,1.590695207
		) (14,2.81090271
		)};
	
	\addplot  +[bar shift=.2cm, area legend]coordinates {
		(1,6.64169266
		) (2,3.990608214
		) (3,2.37224625
		) (4,9.497117698
		) (5,1.306468858
		) (6,5.856035256
		) (7,0.025616501
		) (8,0.38968042
		) (9,-1.21444E-08
		) (10,1.219317389
		) (11,-0.001228333
		) (12,4.794471844
		) (13,0.897680592
		) (14,0.050938448
		)};
	\end{axis}
	\end{tikzpicture}
	\caption{Results Summary for NESTA Instances From Congested Operating Conditions.}
	\label{fig:NESTAbar}
\end{figure*}
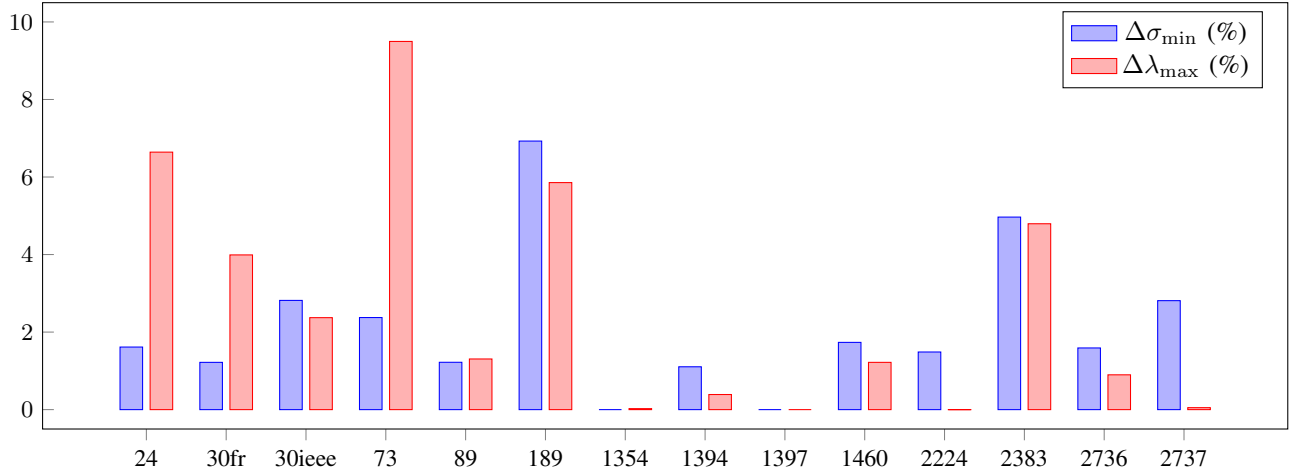

In this section, we present extensive computational results on the proposed VSC-OPF model \eqref{eq:newVSCOPF}, its SOCP relaxation \eqref{eq:VSC-OPF-SOCP}, and the sparse approximation \eqref{eq:sparse_VSC-OPF-SOCP} tested on standard IEEE instances available from \textsc{Matpower} \cite{Zimmerman11} and instances from the \textsc{NESTA} 0.6.0 archive \cite{Coffrin14}. The code is written in \textsc{Matlab}. For all experiments, we used a 64-bit computer with Intel Core i7 CPU 2.60GHz processor and 4 GB RAM. We study the effectiveness of the proposed VSC-OPF on achieving voltage stability, the tightness of the SOCP relaxation for the VSC-OPF, as well as the speed-up and accuracy of the sparse approximation.

Two different solvers are used for VSC-OPF:
\begin{itemize}
	\item Nonlinear interior point solver IPOPT \cite{Wachter06} is used to find local optimal solutions to VSC-OPF.
	\item Conic interior point solver MOSEK 7.1 \cite{mosek13} is used to solve the SOCP relaxation of VSC-OPF.
\end{itemize}

\subsection{Method}

Below we briefly describe the methodologies used in this section to evaluate and demonstrate the effectiveness of the proposed VSC-OPF formulation.

\subsubsection{Evaluating the performance of the proposed VSC-OPF}
During normal operating conditions, the voltage stability condition (\ref{condition}) is normally satisfied. That is, at least for lightly loaded IEEE test cases in \textsc{Matpower}, the constraint \eqref{eq:newVSCOPF:cond} with small $\underline{t}$ will not be binding. This is to be expected, since the stability margins of systems under normal operating conditions are relatively high. To evaluate the formulation in a more meaningful way, we set the margin threshold $\underline{t}$ as follows.

To determine the voltage stability threshold in \eqref{eq:newVSCOPF} for each test instance, we first solve a \emph{minimum threshold maximization problem}. That is, we maximize the minimum value of the left hand side of \eqref{eq:newVSCOPF:cond} among all load buses subject to power flow, nodal voltage, and generation constraints \eqref{eq:ACOPF:realpf} -- \eqref{eq:ACOPF:vol}. The threshold $\underline{t}_i$ in \eqref{eq:newVSCOPF:cond} is set as the slightly decreased maximum threshold from the optimal objective value. In this way, we try to force the voltage stability constraint \eqref{eq:newVSCOPF:cond} to be binding and examine the effect of restraining a high $\underline{t}_i$ on system voltage stability improvement.
	
For comparison, we also solve a \emph{relaxed OPF problem} for each test instance, which is the same as \eqref{eq:newVSCOPF} except that the voltage stability constraint \eqref{eq:newVSCOPF:cond} is unbounded. Two votlage stability indices, i.e. the MSV of the reduced power flow Jacobian $J_{LL}$ and the loading margins to voltage instability of the VSC-OPF formulation \eqref{eq:newVSCOPF} and the relaxed OPF problem are compared. It is expected that constraint \eqref{eq:newVSCOPF:cond} restrains system stability level such that level of stability is improved and voltage stability indices for the VSC-OPF formulation are superior to that of the relaxed OPF problem.

\subsubsection{Recovering bus voltage phasors from SOCP relaxation}

To evaluate the SOCP relaxation \eqref{eq:VSC-OPF-SOCP}, in addition to examine the optimality gap, we compare the MSV obtained by solving \eqref{eq:VSC-OPF-SOCP} with the one obtained from the original problem \eqref{eq:newVSCOPF}, which requires the recovery of nodal voltages. We know the variables $c_{ii}$ are simply the squared bus voltage magnitudes, so bus voltage magnitudes can be directly recovered from SOCP results. To recover voltage phase angles, we use the following relationship:
\begin{equation} \label{eq:thetadiff}
	A_{inc}^T\theta = b
\end{equation}
where $A_{inc}$ is the bus incidence matrix and $b$ is the vector of phase angle differences which can be calculated from SOCP results as $b_k = \atan2 (s_{ij}, c_{ij})$\footnote{$\atan2(y,x) = \begin{cases}
\arctan \frac{y}{x} & x>0 \\ \arctan\frac{y}{x} + \pi & y \ge 0, x<0 \\ \arctan\frac{y}{x} - \pi & y<0, x<0 \\ +\frac{\pi}{2} & y > 0, x = 0 \\ -\frac{\pi}{2} & y < 0, x = 0 \\ \mathrm{undefined} & y = 0, x = 0 \end{cases}$} if $(i,j)$ is the $k$th branch in $\mathcal{E}$. Denote the number of buses by $n_g := n+m$ and number of branches by $n_\ell$, then $n_\ell > n_g$ for almost all meshed networks, and the system \eqref{eq:thetadiff} is overdetermined. We find the least squares solution of \eqref{eq:thetadiff} through pseudoinverse of the bus incidence matrix:
\begin{equation}
	\tilde{\theta} = (A_{inc}^T)^\dagger b.
\end{equation}
Therefore the phase angle of bus voltages can be recovered and the voltage at bus $i$ is given as $V_i = \sqrt{c_{ii}}e^{\mathrm{i}\tilde{\theta}_i}$. The recovered voltages will be used to calculate the MSVs of the SOCP-VSC-OPF results.

\subsection{Results and discussions}
The results of our computational experiments on VSC-OPF and its SOCP relaxation are presented in Table \ref{tb:IEEE} and Fig. \ref{fig:NESTAbar} for standard IEEE and NESTA instances, respectively. The ``Cost'' columns in Table \ref{tb:IEEE} shows the objective values of the VSC-OPF model \eqref{eq:newVSCOPF} and its SOCP relaxation \eqref{eq:VSC-OPF-SOCP}. In addition, six sets of information are provided in Table \ref{tb:IEEE}:
\begin{itemize}
	\item OG(\%) is the percentage optimality gap between the lower bound $LB$ of the objective value obtained from the SOCP relaxation of VSC-OPF \eqref{eq:VSC-OPF-SOCP} and an upper bound $UB$ obtained from \eqref{eq:newVSCOPF} by IPOPT. It is calculated as $100\% \times (1 - LB/UB)$.
	\item $\underline{t}$ is the fixed voltage stability threshold used in the optimization problem (right hand side of \eqref{eq:newVSCOPF:cond}). 
	\item $t_a^{AC}$ is the minimum value of $|V_i| - \sum_{j=1}^n A_{ij}/|V_j|$ for all load bus $i$ calculated \emph{after} solving VSC-OPF \eqref{eq:newVSCOPF}.
	\item $\Delta \lambda_{\max}^{AC} (\%)$ is the percentage increase of loading margins of VSC-OPF \eqref{eq:newVSCOPF} ($\lambda_1$) and that of its relaxed OPF counterpart ($\lambda_2$) calculated as $100\% \times (\lambda_1/\lambda_2 - 1)$. The loading margin is the maximum loading multiplier such that the power flow Jacobian remains nonsingular under proportional load and generation increase. They are calculated using the Continuation Power Flow tool in \textsc{Matpower}.
	\item $\Delta \sigma_{\min}^{AC} (\%)$ is the percentage increase of MSV of the reduced power flow Jacobian of VSC-OPF \eqref{eq:newVSCOPF} ($\sigma_1$) and that of its relaxed OPF counterpart ($\sigma_2$) calculated as $100\% \times (\sigma_1/\sigma_2 - 1)$.
	\item DS(\%) is the percentage difference between the MSV $\sigma_{\min}^{AC}$ obtained from AC-OPF \eqref{eq:newVSCOPF} and the MSV $\sigma_{\min}^{SOCP}$ from the SOCP relaxation. calculated as $100\% \times |\sigma_{\min}^{SOCP} / \sigma_{\min}^{AC} - 1|$.
\end{itemize}

\subsubsection{Stability margin improvement}
As shown by Table \ref{tb:IEEE}, on average, the proposed VSC-OPF model improves the loading margin by about 2\% for the IEEE instances over the relaxed OPF problem with unbounded voltage stability constraint. We also see than several instances have significantly larger improvements. For example, case30, case\_ieee30 of 30-bus system and case39 of 39-bus system all have more than 5\% improved loading margins; it is seen from Fig. \ref{fig:NESTAbar} that several instances in the NESTA archive have more than 5\% loading margin increase as well, for instance 24-bus, 73-bus, and 189-bus systems. It is worth noting that there are two IEEE instances (118-bus and 300-bus systems) where the loading margins decrease. This does not necessarily mean the system voltage stability level is worsen as the loading marin is only measured along a specific ray of loading variation. In fact, the MSVs of the two instances both increase, suggesting the overall stability condition may be improved.

As for the MSV, we see from Table \ref{tb:IEEE} that for IEEE instances the increase are all nonnegative, with an average value of 0.59\%. This is consistent with our discussion in Section \ref{sect:condition} that the voltage stability constraint \eqref{eq:newVSCOPF:cond} helps preserve the diagonal dominance of the transformed power flow Jacobian. In fact, the increase of MSVs for NESTA instances are all nonnegative as well. In addtion, we see from Fig. \ref{fig:NESTAbar} that there are a few instances that experience large MSV increase, notably 189-bus and 2383-bus systems. We also see that there are instances for which both $\Delta \sigma_{\min}$ and $\Delta \lambda_{\max}$ are small, which may indicate that the relaxed OPF problems already yield solutions that have high voltage stability levels.

\subsubsection{Tightness of SOCP relaxation}

Table \ref{tb:IEEE} shows the average optimality gap between the SOCP relaxation \eqref{eq:VSC-OPF-SOCP} and a local solution of the non-convex VSC-OPF \eqref{eq:newVSCOPF} is about 1.45\%. The optimality gap is quite small, but still larger compared with the standard OPF. This can be attributed to the fact that the flow limits for IEEE instances are high and most of them are not binding in standard OPF, while the voltage stability constraints for VSC-OPF are binding in our experiment.

\subsubsection{Effect of sparse approximation}

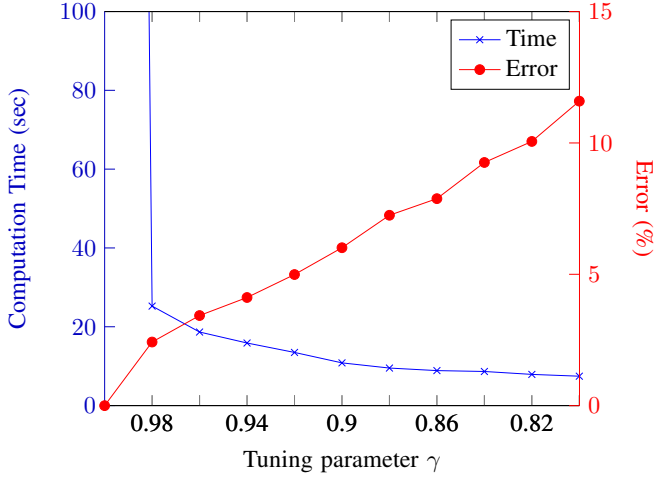
\begin{figure}[!t]
	\centering
	\resizebox{8.9cm}{!}{
	\begin{tikzpicture}
	\pgfplotsset{
		xmin = 0, xmax = 1,
		ymin = 0, ymax = 1200,
		xtick={0.1,0.2,0.3,0.4,0.5,0.6,0.7,0.8,0.9,1},
		xticklabels={0.98,,0.94,,0.9,,0.86,,0.82,},
		y axis style/.style={
			yticklabel style=#1,
			ylabel style=#1,
			y axis line style=#1,
			ytick style=#1
		}
	}
	
	\begin{axis}[
	axis y line*=left,
	y axis style=blue!75!black,
	ymin=0, ymax=100,
	xlabel= Tuning parameter $\gamma$,
	ylabel= Computation Time (sec)
	]
	\addplot[mark=x,blue] 
	coordinates{
		(0.0,1190.98)
		(0.1,25.24157303)
		(0.2,18.65813821) 
		(0.3,15.87233406)
		(0.4,13.48229338)
		(0.5,10.83181464) 
		(0.6,9.517315022)
		(0.7,8.871685079)
		(0.8,8.647633729)
		(0.9,7.907575976)
		(1.0,7.450971003)
	}; \label{Hplot}
	\end{axis}
	
	\begin{axis}[
	hide y axis
	]
	\addlegendimage{/pgfplots/refstyle=Hplot}\addlegendentry{Time}
	\addplot[mark=*,red] 
	coordinates{
		(0,0)
		(0.1,2.421822102/15*1200)
		(0.2,3.430384924/15*1200) 
		(0.3,4.116964053/15*1200)
		(0.4,4.991235847/15*1200)
		(0.5,6.011817768/15*1200)
		(0.6,7.247930408/15*1200)
		(0.7,7.879719321/15*1200)
		(0.8,9.256900766/15*1200)
		(0.9,10.05782037/15*1200)
		(1.0,11.59135124/15*1200)
	}; \addlegendentry{Error}
	\end{axis}
	
	\pgfplotsset{every axis y label/.append style={rotate=180,yshift=9cm}}
	\begin{axis}[
	xmin=0, xmax=1,
	ymin=0, ymax=15,
	hide x axis,
	axis y line*=right,
	ylabel= Error (\%),
	y axis style=red
	]
	\end{axis}
	\end{tikzpicture}}
	\caption{Sparse approximation of NESTA 2737-bus test system}
	\label{fig:sparse}
\end{figure}

The result summary of our computational experiments on the sparse approximation of VSC-OPF for large NESTA instances are presented in Table \ref{tb:NESTA_sparse}. The sparsity parameter in Algorithm \ref{alg:sparsify} is chosen to be 0.98. The ``Time'' columns in the table show the computation time of the VSC-OPF model \eqref{eq:VSC-OPF-SOCP} and the sparse approximation \eqref{eq:sparse_VSC-OPF-SOCP}. In addition, the table reports two sets of data as described below:
\begin{itemize}
	\item DCT(\%) is the percentage time difference between the computation time $ct_n$ of \eqref{eq:VSC-OPF-SOCP} and $ct_s$ of \eqref{eq:sparse_VSC-OPF-SOCP}. It is calculated as $100\% \times (1 - ct_s / ct_n)$.
	\item DC(\%) is the percentage difference between the objective value $c^{SOCP}$ of the model \eqref{eq:VSC-OPF-SOCP} and $c^s$ of \eqref{eq:sparse_VSC-OPF-SOCP} calculated as $100 \% \times |c^s / c^{SOCP} - 1|$.
\end{itemize}

For NESTA systems with less than five buses, the sparse approximation \eqref{eq:sparse_VSC-OPF-SOCP} and the original SOCP relaxation model \eqref{eq:VSC-OPF-SOCP} are exactly the same. For system sizes ranging between 6 buses and 300 buses, the computation times of model \eqref{eq:VSC-OPF-SOCP} are sufficiently short (less than 2 seconds), which render the sparse approximation unnecessary. However, for systems with more than 1000 buses, the sparse approximation brings about significant speed-up. In fact, the speed-ups are above 90\% for all instances with more than 2000 buses and the optimal solutions are obtained in less than 30 seconds for all instances. Our simulation experiments suggest that the solution accuracies are extremely high. For larger systems with more than 1000 buses, the differences of cost between \eqref{eq:VSC-OPF-SOCP} and \eqref{eq:sparse_VSC-OPF-SOCP} are all less than 0.01\%.

\begin{table}[!t]
	\centering
	\caption{Results Summary of Sparse Approximation for Large NESTA Instances From Congested Operating Conditions.}
	\begin{tabular}{ |c||c|c||c|c|}
		\hline
		\multirow{2}{*}{Test Case} & \multicolumn{2}{c||}{Time (sec)} & \multirow{2}{*}{DCT (\%)} & \multirow{2}{*}{DC (\%)} \\
		\cline{2-3}
		& Normal & Sparse & & \\
		\hline
		\hline
		nesta\_case1354\_pegase & $25.04$ & $4.24$ & $83.05$ & $0.00$ 	\\
		\hline
		nesta\_case1394sop\_eir & $39.44$ & $9.75$ & $75.27$ & $0.00$	\\
		\hline
		nesta\_case1397sp\_eir & $40.42$ & $10.34$ & $74.42$ & $0.00$ \\
		\hline
		nesta\_case1460wp\_eir & $39.95$ & $10.54$ & $73.61$ & $0.00$  \\
		\hline
		nesta\_case2224\_edin & $274.68$ & $22.23$ & $91.91$ & $0.00$	\\
		\hline
		nesta\_case2383wp\_mp & $90.54$ & $6.92$ & $92.35$ & $0.00$  \\
		\hline
		nesta\_case2736sp\_mp & $496.59$ & $14.92$ &  $97.00$ & $0.00$	\\
		\hline
		nesta\_case2737sop\_mp & $1190.98$ & $25.24$ &  $97.04$ & $0.00$	\\
		\hline
		\textbf{average} & \bm{$274.71$} & \bm{$13.02$} & \bm{$85.58$} & \bm{$0.00$} \\
		\hline
	\end{tabular}
	\label{tb:NESTA_sparse}
\end{table}

Fig. \ref{fig:sparse} presents the results of our computational experiments on the sparse approximation of VSC-OPF for NESTA 2737-bus test instance with varying $\gamma$. This specific test instance is chosen since it is the largest instance we have experimented with and also the one that takes the longest computation time. In fact, it takes almost 1200 seconds to compute the optimal solution for the test instance. In the figure, blue cross shows the computation time and red dot shows the relative error of the MSV results. The relative error is calculated as $|\sigma_1 - \sigma_{\gamma}| / |\sigma_1 - \sigma_0|$ where $\sigma_1$, $\sigma_0$ and $\sigma_{\gamma}$ are the MSVs given by SOCP relaxation \eqref{eq:VSC-OPF-SOCP} ($\gamma = 1$), relaxed OPF problem ($\gamma = 0$), and sparse approximation with  tuning paramter $\gamma$. For this test instance, $\sigma_1 \approx 0.451$ and $\sigma_0 \approx 0.439$, it can be seen from Fig. \ref{fig:NESTAbar} that there is an approximately 3\% increase from $\sigma_0$ to $\sigma_1$. We see from Fig. \ref{fig:sparse} that the computation time sees a drastic decrease with a very small deviation of the tuning parameter from 1. Even with $\gamma$ as high as 0.98, the computation time can be reduced to within 30 seconds, while further decrease in $\gamma$ reduces the computation time but to a lesser extent, and the computation time gradually stabilizes at around 10 seconds. The relative error increases almost linearly with the decrease of $\gamma$. For $\gamma = 0.98$, the relative error is only around 3\%.

\subsection{Comparison with alternative VSC-OPF formulation}

In this section, we compare the proposed VSC-OPF formulation with an alternative formulation proposed in \cite{Kamwa14}. The VSC-OPF employed in \cite{Kamwa14} is based on Voltage Collapse Proximity Index (VCPI), a voltage stability index quantifying power transfer margins of individual branches. The VCPI index for a branch is defined as
\begin{equation}
\mathrm{VCPI} = \frac{P_r}{P_{r,\max}},
\end{equation}
where $P_r$ is the real power transferred to the receiving end, $P_{r,\max}$ is the maximum real power that can be transferred to the receiving end assuming the voltage at the sending end is fixed. It is known from the definition that $0 \le \mathrm{VCPI} \le 1$ and high VCPI signifies a system that is more stressed. Let the sending and receiving end bus voltages be $|V_s|e^{\mathrm{i}\theta_s}$ and $|V_r|e^{\mathrm{i}\theta_r}$, and let $V_d := |V_s|e^{\mathrm{i}\theta_s} - |V_r|e^{\mathrm{i}\theta_r}$, then the index can be represented by the two voltages as
\begin{equation}
\mathrm{VCPI} = \frac{2|V_r||V_d|}{|V_s|^2} + 2\frac{|V_r|\cos(\theta_s - \theta_r)}{|V_s|} - 2 \frac{|V_r|^2}{|V_s|^2} \label{eq:vcpi}
\end{equation}
The resulting VSC-OPF formulation is the same as \eqref{eq:newVSCOPF} except that the constraint \eqref{eq:newVSCOPF:cond} is replaced with $\mathrm{VCPI}_{\max} \le \mathrm{VCPI}_{\mathrm{limit}}$ where $\mathrm{VCPI}_{\max}$ is the maximum VCPI among all branches and $\mathrm{VCPI}_{\mathrm{limit}}$ is a preset threshold. We would like to point out that the VCPI index is heuristic in nature since it has been shown in \cite{Grijalva1} that maximum branch flows are generally encountered well before the onset of voltage instability.

The results of \eqref{eq:newVSCOPF} as well as that of the above formulation based on VCPI depend on the preset threshold. It is difficult to choose comparable thresholds for the two indices that represent similar system stress levels, since after all the effect of the indices in reflecting system stress level is what we want to investigate. Therefore, we propose to compare the two indices by formulating the `voltage stability improvement' problem in \cite{Kamwa14}. That is, instead of using the voltage stability index as a constraint, we directly optimize the sum of stability indices, subject to power flow equations, nodal voltage and power generation constraints \eqref{eq:ACOPF:realpf} -- \eqref{eq:ACOPF:vol}. We then denote the two optimization problems as $(\mathrm{P}_{\mathrm{VCPI}})$ and $(\mathrm{P}_C)$, since they optimize the sum of VCPI and C-index (\cite{Wang17}), respectively.

One thing we notice with $(\mathrm{P}_{\mathrm{VCPI}})$ is that for almost all instances, the problems experience very slow convergence: they do not converge after 1,000 iterations in IPOPT. This is probably due to the poor numerical properties of the VCPI index \eqref{eq:vcpi}, since the gradient and hessian of the constraint involve the reciprocal of $V_d$, which is almost zero when $V_s$ and $V_r$ are close. We end up with the \textsc{Matlab} function \texttt{fmincon} with interior-point solver as was used in \cite{Kamwa14} for $(\mathrm{P}_{\mathrm{VCPI}})$ with the default maximum number of iterations of 3,000. The program terminates with a feasible solution that is best possible, to which we compare with $(\mathrm{P}_C)$ solved with IPOPT. The results are shown in Table \ref{tb:compare}. It is seen that the proposed approach outperforms the one in \cite{Kamwa14}.

\begin{table}[!t]
	\centering
	
	\caption{Comparison of the effect of voltage stability improvement of different VSC-OPF formulations.}
	\begin{tabular}{ |c||c|c||c|c|}
		\hline
		\multirow{2}{*}{Test Case} & \multicolumn{2}{c||}{$\Delta \lambda_{\max}$ (\%)} & \multicolumn{2}{c|}{$\Delta \sigma_{\min}$ (\%)} \\
		\cline{2-5}
		& $(\mathrm{P}_C)$ & $(\mathrm{P}_{\mathrm{VCPI}})$ & $(\mathrm{P}_C)$ & $(\mathrm{P}_{\mathrm{VCPI}})$ \\
		\hline
		\hline
		case24\_ieee\_rts & $9.87$ & $5.21$ &  $0.32$ & $-0.32$ \\
		\hline
		case30 & $15.06$ & $-0.59$ &  $0.01$ & $-5.13$ \\
		\hline
		case\_ieee30 & $9.02$ & $6.02$ &  $4.05$ & $3.38$ \\
		\hline
		case39 & $8.96$ & $1.39$ & $0.51$ & $-4.01$ \\
		\hline
		case57 & $0.52$ & $-1.66$ &  $0.52$ & $-0.76$ \\
		\hline
		case89pegase & $1.39$ & $0.99$ &  $1.27$ & $-5.99$ \\
		\hline
		case118 & $38.43$ & $27.57$ &  $1.58$ & $-4.70$ \\
		\hline
		case300 & $3.33$ & $1.63$ &  $3.43$ & $-2.37$ \\
		\hline
		case1354pegase & $0.92$ & $-4.08$ &  $0.05$ & $-4.69$ \\
		\hline
		case2383 & $2.64$ & ---\textsuperscript{$\dagger$} &  $0.80$ & $-1.63$ \\
		\hline
		\textbf{average} & \bm{$9.01$} & \bm{$4.05$}\textsuperscript{$\ddagger$} & \bm{$1.25$} & \bm{$-2.62$}  \\
		\hline
		\multicolumn{5}{l}{\textsuperscript{$\dagger$}\footnotesize{Power flow based on optimization solution does not converge }} \\
		\multicolumn{5}{l}{\textsuperscript{$\ddagger$}\footnotesize{Average over the first nine cases }} \\
	\end{tabular}
	\label{tb:compare}
\end{table}

\section{Conclusions} \label{sect6}
We have presented a sufficient condition for power flow Jacobian nonsingularity and shown that the condition characterizes a set of voltage stable solutions. A new VSC-OPF model has been proposed based on the sufficient condition. By using the fact that the load powers are constant in an OPF problem, we reformulate the voltage stability condition to a set of second-order conic constraints in a transformed variable space. Furthermore, in the new variable space, we have formulated an SOCP relaxation of the VSC-OPF problem as well as its sparse approximation. Simulation results show that the proposed VSC-OPF and its SOCP relaxation can effectively restrain the stability stress of the system; the optimality gap of the SOCP relaxation is slightly larger than that of the standard OPF problem on IEEE instances due to tightness of the constraints; and the sparse approximation yields significant speed-up on larger instances with small accuracy compromise. It has also been shown that the proposed method outperforms existing one in terms of effectiveness and computational properties.





\end{document}